\documentclass[11pt]{amsart}
\usepackage{amsmath,amssymb,amsthm,latexsym,geometry,pifont,fullpage}
\usepackage[all]{xy} 

\newcommand{\R}{\mathbb{R}}
\newcommand{\Rn}{\R^n}
\newcommand{\ttw}{\mathbb{T}^2}
\newcommand{\tn}{\mathbb{T}^n}

\newcommand{\glthr}{\mathrm{GL}(3,\Z)}

\newcommand{\cohot}{\mathrm{H}^2(B; \Z^n_{\rho})}
\newcommand{\coho}{\mathrm{H}}

\newcommand{\cohoz}{\mathrm{H}^2(B; \Z)}

\newcommand{\cohor}{\mathrm{H}^3(B;\R)}

\newcommand{\Z}{\mathbb{Z}}
\newcommand{\cotan}{\mathrm{T}^*}
\newcommand{\de}{\mathrm{d}}
\newcommand{\fib}{\xymatrix@1{ F \;\ar@{^{(}->}^-{\iota}[r] & (M, \omega)
    \ar[r]^-{\pi} & B}}
\newcommand{\fibs}{\xymatrix@1{ F \;\ar@{^{(}->}[r]& M
    \ar[r] & B}}
\newcommand{\fibpl}{\xymatrix@1{ \mathrm{H}_1(F,\Z) \;\ar@{^{(}->}[r]
    & \mathcal{P} 
    \ar[r] & B}}
\newcommand{\fibrpl}{\xymatrix@1{ \mathrm{H}_1(F,\Z) \;\ar@{^{(}->}[r]
    & \mathcal{P}' 
    \ar[r] & B}}
\newcommand{\affr}{\mathrm{Aff}_{\Z}(\Rn)}
\newcommand{\afft}{\mathrm{Aff}_{\Z}(\Rn/\Z^n)}
\newcommand{\gln}{\mathrm{GL}(n,\Z)}
\newcommand{\zn}{\Z^n}
\newcommand{\tth}{\mathbb{T}^3}
\newcommand{\zt}{\mathbb{Z}^3}
\newcommand{\rt}{\mathbb{R}^3}

\newcommand{\am}{(B,\mathcal{A})}
\newcommand{\lf}{\tn \hookrightarrow (M,\omega) \to \am}
\newcommand{\rzn}{\R^n/\Z^n}

\theoremstyle{definition}
\newtheorem{defn}{Definition}

\newtheorem*{mr}{Main Result}

\theoremstyle{remark}
\newtheorem{rk}{Remark}
\newtheorem*{qn}{Question}
\newtheorem*{notation}{Notation}

\theoremstyle{plain}
\newtheorem{thm}{Theorem}
\newtheorem{lemma}{Lemma}
\newtheorem{cor}{Corollary}

\title{Almost Lagrangian Obstruction} 
\author{Daniele Sepe}
\thanks{I would like to thank Jarek Kedra for many stimulating
  discussions and the referee for his insightful comments.}
\address{School of Mathematics and Maxwell Institute for Mathematical
  Sciences, The University of Edinburgh, James 
  Clerk Maxwell Building, King's Buildings, Edinburgh, EH9 3JZ, UK}
\email{ds316@le.ac.uk}
\begin{document}
\maketitle
\begin{abstract}
  The aim of this paper is to describe the obstruction for an almost
  Lagrangian fibration to be Lagrangian, a problem which is central to the
  classification of Lagrangian fibrations and, more generally, to
  understanding the obstructions to carry out surgery of integrable
  systems, an idea introduced in \cite{zun_symp2}. It is shown that
  this obstruction (namely, the
  homomorphism $\mathcal{D}$ of Dazord and Delzant
  \cite{daz_delz} and Zung \cite{zun_symp2}) is related to the cup
  product in cohomology with local coefficients on the base space $B$ of the
  fibration. The map is described explicitly and some explicit examples
  are calculated, thus providing the first examples of non-trivial
  Lagrangian obstructions.
\end{abstract}
\setcounter{tocdepth}{1}
\tableofcontents

\section{Introduction}
A fibration $F \hookrightarrow M \to B$ is said to be
\emph{Lagrangian} if $M$ admits a symplectic form $\omega$ such that
the fibres are Lagrangian. Throughout this paper, the fibres $F$ are
taken to be compact and connected; by the Liouville-Mineur-Arnol`d
theorem (cf. \cite{arnold}), this condition implies that the fibres
are diffeomorphic to tori. Another consequence of this theorem is that
the base space of a Lagrangian fibration is an integral affine
manifold, as shown by many authors (\textit{e.g.}
\cite{bates_monchamp,dui,sepe_topc,zun_symp2}). An integral affine
structure $\mathcal{A}$ on $B$ is an atlas $\{(U_i, \psi_i)\}$ whose
changes of coordinates $\psi_i \circ \psi_j^{-1} : \Rn \to \Rn$ lie in the group
$$\affr : = \gln \ltimes \Rn, $$
\noindent
where the action of $\gln$ on $\Rn$ is the standard one. 

\begin{notation}
  An integral affine manifold is henceforth denoted by $\am$, while $\tn
  \hookrightarrow (M,\omega) \to \am$ denotes a Lagrangian fibration
  which induces the integral affine structure $\mathcal{A}$ on its
  base space. The latter is referred to as a Lagrangian fibration over
  $\am$.   
\end{notation}

Given a Lagrangian fibration over $\am$, the atlas $\mathcal{A}$ on the
$n$-dimensional manifold $B$ determines a representation 
$$ \mathfrak{a} : \pi_1(B) \to \affr,$$
\noindent
which arises from identifying $\pi_1(B)$ as the group of deck
transformations acting by integral affine diffeomorphisms on the
universal cover $(\tilde{B},\tilde{\mathcal{A}})$, where
$\tilde{\mathcal{A}}$ is the integral affine structure on $\tilde{B}$
induced by the universal covering $q: \tilde{B} \to B$ (cf. \cite{aus}). The homomorphism
$\mathfrak{a}$ is known as the affine holonomy of the integral affine 
manifold $\am$, while the composite
$$ \mathfrak{l} := \mathrm{Lin} \circ \mathfrak{a} : \pi_1(B) \to
\affr \to \gln $$
\noindent
is called the linear holonomy of $\am$. On the other hand, associated
to a Lagrangian fibration over $\am$ is a homomorphism
$$ \rho: \pi_1(B) \to \gln,$$
\noindent
called the \emph{monodromy} of the fibration (cf. \cite{dui}). The
monodromy and linear holonomy are related by

\begin{equation}
  \label{eq:27}
  \mathfrak{l} = \rho^{-T},
\end{equation}
\noindent
where $-T$ denotes inverse transposed. Conversely, any $n$-dimensional
integral affine manifold $\am$ with linear holonomy $\mathfrak{l}$ is
the base space of a Lagrangian fibration $\lf$ with monodromy
representation $\mathfrak{l}^{-T}$. \\

Another consequence of the Liouville-Mineur-Arnol`d theorem also is that the
structure group of a Lagrangian fibration over an $n$-dimensional
manifold reduces to
$$ \afft := \gln \ltimes \rzn.$$
\noindent 
This observation implies there are only two
topological invariants associated to a Lagrangian fibration, namely
the aforementioned monodromy $\rho$, and a cohomology class $[c] \in
\cohot$ called the Chern class, which is the obstruction to the
existence of a section (cf. \cite{dui,sepe_topc}). Note that $[c]$
lies in the cohomology theory with local coefficients in the 
module $\pi_1(\tn)\cong \Z^n$ twisted by the representation $\rho$ (cf.
\cite{daz_delz,dui,sepe_topc,zun_symp2}). Given an integral affine
manifold $\am$ with linear holonomy $\mathfrak{l}$, the set of
(isomorphism classes of) Lagrangian 
fibrations over $\am$ (and, thus, with monodromy $\mathfrak{l}^{-T}$)
is in $1$-$1$ correspondence with a subgroup 
$$ R \leq \coho^2(B;\Z^n_{\mathfrak{l}^{-T}}). $$
\noindent
It is natural to ask what subgroup $R$ is: work of Dazord and Delzant
in \cite{daz_delz} and of Zung in \cite{zun_symp2} proves that there exists a
homomorphism
$$ \mathcal{D} : \coho^2(B;\Z^n_{\mathfrak{l}^{-T}}) \to \cohor $$
\noindent
with $R = \ker \mathcal{D}$. This map represents the obstruction to
construct an appropriate symplectic form on the total space of an
\emph{almost Lagrangian fibration} over $\am$ (cf. Definition \ref{defn:suit}). However, there is no explicit general description
of this homomorphism in the literature. \\ 

The aim of this paper is to relate the
homomorphism $\mathcal{D}$ to the cohomology of $B$ in the case of
regular Lagrangian fibrations and to carry out some explicit
calculations. Dazord and Delzant in \cite{daz_delz} show that, when
the monodromy representation is trivial, the homomorphism
$\mathcal{D}$ is given by the \emph{cup product} on the singular
cohomology ring of $B$. The main result of this paper generalises this
description.  

\begin{mr}
  The homomorphism $\mathcal{D}$ arises from taking twisted cup
  product on $B$ (\textit{i.e.} cup product in cohomology
  with local coefficients).
\end{mr}

For a fixed integral affine manifold $\am$, the homomorphism
$\mathcal{D}$ distinguishes those almost Lagrangian fibrations 
over $\am$ which are actually Lagrangian and those which are not; the
latter are called \emph{fake}. The local
structure of almost and fake Lagrangian fibrations has been studied
by Fass{\`o} and Sansonetto \cite{fasso} in the context of a generalised
notion of Liouville integrability. Elements of $R$ are called
\emph{realisable}, the terminology 
coming from the theory of symplectic realisations of Poisson manifolds
(cf. \cite{daz_delz,vaisman}). The above Main Result is important in the
classification of Lagrangian fibration over manifolds of dimension
greater than or equal to $3$, in so far as it provides an algorithm to
compute the subgroup of realisable classes. Furthermore, the
constructions in this paper are related to the idea of
\emph{integrable surgery}, introduced in \cite{zun_symp2},  
where two or more completely integrable Hamiltonian systems are
glued together to yield a new one. \\

The structure of the paper is as follows. The notion of
\emph{almost Lagrangian fibrations} over an integral affine manifold
$\am$ is introduced in Section
\ref{sec:almost-lagr-fibr}, following ideas in \cite{fasso}. These
fibrations are the right  `candidates' to be Lagrangian, as they share
with their Lagrangian counterparts all topological invariants and
constructs. Section \ref{sec:general} deals with the proof of the
Main Result and is divided into two parts. The first sets up some
notation and describes the obstruction in the case of trivial linear holonomy
of $\am$ (cf. \cite{daz_delz}). The general case is
presented in Section \ref{sec:general-case}: the
homomorphism $\mathcal{D}$ is described in detail and shown
to coincide with the twisted cup product on $B$. Moreover, an explicit
algorithm to compute $R$ in general is provided. This
is a non-trivial task since the homomorphisms involved 
in the definition of $\mathcal{D}$ are defined on simplicial spaces,
which are not easily dealt with from a computational point of view. In
Section \ref{sec:applications}, there are some examples related to the
classification of Lagrangian fibrations. In particular, the algorithm
outlined in Section \ref{sec:general} is
applied to some concrete examples of integral affine manifolds.  


\section{Almost Lagrangian fibrations} \label{sec:almost-lagr-fibr}
Throughout this section, fix an integral affine manifold $\am$ with
linear holonomy
$\mathfrak{l} : \pi_1(B) \to \gln$. Denote by $\pi :M \to B$ the projection
map of a Lagrangian fibration $\lf$. The
Liouville-Arnol`d-Mineur theorem (cf. \cite{arnold}) implies that there
exists a good (in the sense of Leray) open cover $\mathcal{U} = \{U_i\}$ of $B$ by
trivialising neighbourhoods such that $\pi^{-1}(U_i)$ has coordinates
$(\mathbf{x}_i, \mathbf{t}_i)$ (called local \emph{action-angle}
coordinates) with the following properties
\begin{itemize}
\item the action coordinates $\mathbf{x}_i$ define the integral affine
  structure $\mathcal{A}$, while the angle coordinates $\mathbf{t}_i$
  define integral affine coordinates on the fibres;
\item for each $i$, the restriction $\omega_i = \omega|_{\pi^{-1}(U_i)}$
takes the form 
$$ \omega_i = \sum\limits^n_{l=1} \de x_i^l \wedge \de t_i^l; $$
\item the transition functions $\phi_{ji}$ are given by
\begin{equation}
  \label{eq:1}
  \begin{split}
    \phi_{ji} :(U_i \cap U_j) \times \tn & \to (U_i \cap U_j) \times
    \tn \\
    (\mathbf{x}_{i},\mathbf{t}_i) & \mapsto (A_{ji} \mathbf{x}_i + 
    \mathbf{c}_{ji}, A_{ji}^{-T} \mathbf{t}_i +
    \mathbf{g}_{ji}(\mathbf{x}_i))
  \end{split}
\end{equation}
\noindent
where $A_{ji} \in \gln$ and $\mathbf{c}_{ji} \in \Rn$ are constant, and
$\mathbf{g}_{ji} : U_{ji} \to \tn$ is a local function constrained by
the fact that
$$\phi^*_{ji}\omega_j =\omega_i$$
\noindent
for all $i,j$. 
\end{itemize}

The above necessary conditions for a Lagrangian fibration over $\am$
and the results of \cite{fasso} motivate the following definition.

\begin{defn} \label{defn:suit}
  A fibration $\tn \hookrightarrow M \to \am$ is said to be \emph{almost
    Lagrangian} if there exists a good open cover $\mathcal{U} =
  \{U_i\}$ of $B$ such that
  \begin{itemize}
  \item there exist trivialisations $ \varphi_i : \pi^{-1}(U_i) \to U_i
    \times \tn$ inducing local coordinates $(\mathbf{x}_i,
    \mathbf{t}_i)$;
  \item the local coordinates $\mathbf{x}_i$ are integral affine
    coordinates on $\am$;
  \item the transition functions $\phi_{ji} = \varphi_j \circ
    \varphi_i^{-1}$ are of the form 
    \begin{equation}
      \label{eq:laac}
      \phi_{ji}(\mathbf{x}_{i},\mathbf{t}_i) = (A_{ji} \mathbf{x}_i +
      \mathbf{c}_{ji}, A_{ji}^{-T} \mathbf{t}_i +
      \mathbf{g}_{ji}(\mathbf{x}_i)) 
    \end{equation}
    \noindent
    where $A_{ji} \in \gln$ and $\mathbf{c}_{ji} \in \Rn$ are constant,
    $\mathbf{g}_{ji} : U_{ji} \to \tn$ is a local function.
  \end{itemize}
\end{defn}

\begin{rk}\label{rk:1}
  \mbox{}
  \begin{itemize}
  \item An almost Lagrangian fibration $\tn \hookrightarrow M \to \am$
    has structure group $\afft$. Thus the topological classification
    of almost Lagrangian fibrations coincides with that of Lagrangian
    fibrations (cf. \cite{dui}) and it makes sense to consider the
    monodromy and Chern class of such fibrations; 
  \item If $\rho$ denotes the monodromy representation of an almost
    Lagrangian fibration $\tn \hookrightarrow M \to \am$, then
    $$\rho = \mathfrak{l}^{-T},$$
    \noindent
    where $\mathfrak{l}$ denotes the linear holonomy of $\am$;
  \item For a fixed almost Lagrangian fibration, each
    $\pi^{-1}(U_i)$ admits  a local symplectic form 
    \begin{equation}
      \label{eq:19}
      \omega_i =  \sum\limits^n_{l=1} \de x_i^l \wedge \de t_i^l,
    \end{equation}
    \noindent
    where $(\mathbf{x}_i, \mathbf{t}_i)$ are local coordinates as in
    Definition \ref{defn:suit}. An almost Lagrangian fibration is
    Lagrangian if and only if it is possible to choose the transition
    functions $\phi_{ji}$ so that 
    $$ \phi_{ji}^*\omega_j = \omega_i. $$
  \item In analogy with the Lagrangian case, $\tn
    \hookrightarrow M \to \am$ denotes an almost Lagrangian fibration
    constructed from the integral affine structure $\mathcal{A}$ on
    $B$. Such fibrations are in 1-1 correspondence with the set of
    $\tn$-bundles with 
    structure group $\afft$ whose 
    monodromy equals $\mathfrak{l}^{-T}$.
  \end{itemize}
\end{rk}

The period lattice bundle associated to an integral affine manifold
$\am$ plays an important role in the study and construction of almost
Lagrangian fibrations. 

\begin{defn}\label{defn:plb}
  Let $\mathbf{x}_i$ be local integral affine coordinates defined on
  subsets $U_i$ of $\am$. The \emph{period
    lattice bundle} $P_{\am} \subset \cotan 
  B$ associated to $\am$ is the discrete subbundle (with fibre $\Z^n$)
  which is locally 
  defined by
  $$ P_{\am}|_{U_i} := \{ (\mathbf{x}_i, \mathbf{v}_i) \in \cotan U_i:\,
  \mathbf{v}_i \in \Z\langle \de x^1_i, \ldots, \de x^n_i \rangle \}. $$
\end{defn}

\begin{rk}\label{rk:isomorphism_discrete}
  Let $\tn \hookrightarrow M \to \am$ be an almost Lagrangian
  fibration. Just as in the Lagrangian case, the bundle obtained by
  replacing the torus fibres with their first homology groups
  $\coho_1(\tn;\Z)$ is isomorphic to the period lattice bundle
  associated to $\am$. The isomorphism is defined locally using the
  symplectic forms $\omega_i$ of Remark \ref{rk:1}, and it sends the
  differentials $\de x^l_i$ to the homology cycle generated by the
  time-1 flow of the Hamiltonian vector field of the function $x^l_i$
  (cf. \cite{dui}).
\end{rk}

The inclusion $ P_{\am} \hookrightarrow \cotan B$ defines an injective
morphism of sheaves 
\begin{equation}
  \label{eq:9}
  \mathcal{P}_{\am} \hookrightarrow \mathcal{Z}^1(B)
\end{equation}
\noindent
where $\mathcal{P}_{\am}, \mathcal{Z}^1(B)$ denote the sheaves of sections
of the period lattice bundle $P_{\am} \to B$ and of closed sections of the cotangent
bundle $\cotan B \to B$ respectively
(cf. \cite{daz_delz,zun_symp2}). Recall that the Chern class of an
almost Lagrangian fibration over $\am$ lies in the group
$$ \coho^2(B;\mathcal{P}_{\am}) \cong
\coho^2(B;\Z^n_{\mathfrak{l}^{-T}}), $$
\noindent
as shown in \cite{daz_delz,dui}. It is important to note that the
above isomorphism is induced by the isomorphism of $\Z^n$-bundles
described in Remark \ref{rk:isomorphism_discrete}. Henceforth, the
local coefficient system with values in $\Z^n$ and twisted by the
representation $\mathfrak{l}^{-T}$ is identified with the stalk of the
sheaf of sections of the period lattice bundle $P_{\am}$. \\

The morphism of equation \eqref{eq:9} induces a homomorphism of
cohomology groups
\begin{equation}
  \label{eq:23}
  \coho^2(B;\mathcal{P}_{\am}) \to \coho^2(B;\mathcal{Z}^1(B)).
\end{equation}
\noindent
If $\mathcal{C}^{\infty}(B)$ denotes the sheaf of smooth functions on
$B$, there is a short exact sequence of sheaves
$$ \xymatrix@1{ 0 \ar[r] & \R \ar[r] & \mathcal{C}^{\infty}(B)
  \ar[r]^-{\de} & \mathcal{Z}^1(B) \ar[r] & 0}, $$
\noindent
where $\de$ denotes the standard exterior differential; since $\mathcal{C}^{\infty}(B)$ is a fine sheaf, the induced long
exact sequence in cohomology yields an isomorphism
\begin{equation}
  \label{eq:16}
  \coho^{2}(B;\mathcal{Z}^1(B)) \cong \coho^{3}(B;\R).
\end{equation}

\begin{defn}[Dazord and Delzant \cite{daz_delz}]\label{defn:hmdaz_delz}
  The composition of the homomorphism of equation \eqref{eq:23} and
  the isomorphism of equation \eqref{eq:16} yields a homomorphism
  $$ \mathcal{D} : \coho^2(B;\mathcal{P}_{\am}) \to \coho^3(B;\R), $$
  \noindent
  called the \emph{Dazord-Delzant} homomorphism.
\end{defn}

\begin{rk}[Dazord and Delzant \cite{daz_delz}] \label{rk:D}
  For an almost Lagrangian fibration over $\am$ with transition
  functions $\phi_{ji}$ as in Definition \ref{defn:suit}, a \v Cech
  cocycle representing $\mathcal{D}[c]$ is given by
  $$ \kappa_{ji} = \phi_{ji}^* \omega_j - \omega_i. $$
\end{rk}

Remark \ref{rk:D} lies at the heart of the proof of the following
theorem, stated below without proof as it is a known result.

\begin{thm}[Dazord and Delzant \cite{daz_delz}]
  The subgroup of realisable Chern classes for a given integral affine
  manifold $\am$ is given by $\ker \mathcal{D}$.
\end{thm}


\section{The Dazord-Delzant homomorphism and equivariant cup
  products} \label{sec:general}
\subsection{The case with trivial monodromy}
Let $\am$ be an integral affine manifold with trivial linear
holonomy. The period lattice bundle $P_{\am}$ admits a global frame of
closed forms $\{\theta^1, \ldots,
\theta^n\}$ which, locally, are given by the differentials of integral
affine coordinates (cf. \cite{aus}). The set of isomorphism classes of
almost Lagrangian bundles over $\am$ is in 1-1 correspondence with
elements of the cohomology group $\coho^2(B;\Z^n)$. The universal
coefficient theorem and the fact 
that $\Z^n$ is a free module imply that there is an isomorphism
\begin{equation}
  \label{eq:24}
  \coho^2(B;\Z^n) \cong \coho^2(B;\Z) \otimes \Z^n
\end{equation}
\noindent
(cf. \cite{spanier}). The above coefficient system $\Z^n$ is given by
the stalk of the sheaf of sections $\mathcal{P}_{\am}$ of the period lattice
bundle $P_{\am}$, which, in this case, corresponds to integral linear
combinations of $\theta^1, \ldots, \theta^n$. Thus elements of
$\coho^2(B;\Z^n)$ are of the form
$$ \sum\limits_{p,q} n_{pq} [\beta_p] \otimes \theta^q,$$
\noindent
where $n_{pq} \in \Z$ and $\{[\beta_p]\}$ form a basis of
$\coho^2(B;\Z)$. The following theorem, stated below without proof,
describes the Dazord-Delzant homomorphism in the case of trivial
linear holonomy of the base space $\am$.

\begin{thm}[Dazord-Delzant \cite{daz_delz}] \label{thm:main}
The map $\mathcal{D} : \mathrm{H}^2(B; \zn) \to \mathrm{H}^3(B; \R)$
is given by
\begin{equation}
\label{eq:main}
\sum\limits_{p,q}{n_{pq} [\beta_p] \otimes \theta^q} \mapsto
\Bigg[\sum\limits_{p,q}{n_{pq} \beta^{\R}_p \cup \theta^q}\Bigg] 
\end{equation}
\noindent
where $\beta_p^{\R}$ is a $2$-cocycle with values in the real numbers
representing the image of a cocycle $\beta_p$ representing $[\beta_p]$
under the map $\cohoz \to \cohoz \otimes_{\Z} \R \cong \mathrm{H}^2(B; \R)$. 
\end{thm}

\subsection{The general case}\label{sec:general-case}
Theorem \ref{thm:main} shows that the map $\mathcal{D}$ is related to
cup product in the real cohomology of an integral affine manifold
$\am$ with trivial linear holonomy. In general, the situation is not
as simple. Let $\am$ be an
integral affine manifold with linear holonomy $\mathfrak{l}$;
$\coho^2(B;\Z^n_{\mathfrak{l}^{-T}})$ denotes the set of isomorphism
classes of almost Lagrangian fibrations over $\am$. The universal
coefficient theorem for cohomology with local 
coefficients in \cite{greenblatt} gives the following short exact
sequence 
\begin{equation*}
0 \to \mathrm{H}^2(B;\Z [\pi]_{\text{taut}}) \otimes_{\Z [\pi]} \zn \to
\cohot \to \mathrm{Tor}(\mathrm{H}^3(B;\Z [\pi]_{\text{taut}}), \zn) \to
0 
\end{equation*}
\noindent
where $\pi = \pi_1(B)$, $\Z [\pi]$ denotes the group ring of $\pi$ and
$\text{taut}: \pi \to \mathrm{Aut}(\Z[\pi])$ defines the tautological
representation of $\pi$ on $\Z [\pi]$ given by left multiplication. In
general, neither $\mathrm{H}^3(B;\Z [\pi]_{\text{taut}})$ nor $\zn$ are
free $\Z [\pi]$-modules and so there is no equivariant equivalent of the
isomorphism in equation \eqref{eq:24}. In what follows, a candidate for the map
$\mathcal{D}$ is suggested and then it is proved to be the correct one
in Theorem \ref{thm:equimain}.\\ 

Throughout this section, fix an integral affine manifold $\am$ with
linear holonomy $\mathfrak{l}$ and let $\pi$ denote its fundamental
group. Let $\tn \hookrightarrow M \to \am$ be
an almost Lagrangian fibration with Chern class $[c]
\in \coho^2(B;\Z^n_{\mathfrak{l}^{-T}})$. Let $q: \tilde{B} \to B$ be
the universal covering of $B$, and let $\tilde{\mathcal{A}}$ denote
the induced integral affine structure on $\tilde{B}$ via $q$
(cf. \cite{aus}).

\begin{rk}\label{rk:cc}
The Chern class $[c]$ of the fibration $M \to B$ is pulled back to the
Chern class $[\tilde{c}]=q^{*} [c] \in \mathrm{H}^2(\tilde{B};\zn)$ of
$q^* M \to (\tilde{B},\tilde{A})$ by
functoriality of the Chern class (cf. \cite{sepe_topc}). A cocycle representing
$[c]$ is a cochain in
$\hom_{\Z[\pi]}(\mathrm{C}_2(\tilde{B}); \Z^n)$, where $\Z^n$ is a
$\Z[\pi]$-module via the monodromy representation
$\mathfrak{l}^{-T}$. In particular, if $c$
is a such a cocycle, it is just an
\emph{equivariant} cocycle representing $[\tilde{c}]$, since
$\hom_{\Z[\pi]}(\mathrm{C}_*(\tilde{B}); \Z^n)$ is a subcomplex of
complex $\hom_{\Z}(\mathrm{C}_*(\tilde{B}); \Z^n)$ (cf. \cite{white}). 
\end{rk}

Let $P_{\am}, P_{(\tilde{B}, \tilde{\mathcal{A}})}$ denote the period
lattice bundles associated to $\am$ and $(\tilde{B},
\tilde{\mathcal{A}})$ respectively; there is a commutative diagram 
\begin{equation}
\label{eq:commdiag}
\xymatrix{\tilde{P}_{(\tilde{B},\tilde{\mathcal{A}})} \,\ar[d]^-{Q}
  \ar@{^{(}->}[r] & \cotan \tilde{B} \ar[d]^-{Q} \\ 
  P_{\am}\, \ar@{^{(}->}[r] & \cotan B,}
\end{equation}
\noindent
where $Q: \cotan \tilde{B} \cong q^* \cotan B \to \cotan B$ is induced
by the quotient map $q$. The above diagram implies that there exists a
$\pi$-equivariant global frame of closed forms $\{\tilde{\theta}^1, \ldots,
\tilde{\theta}^n\}$ of $\cotan \tilde{B}$; thus the fibre $\zn$ of the
period lattice bundle $P_{\am} \to B$ can be identified with the fibre of
$\tilde{P}_{(\tilde{B},\tilde{\mathcal{A}})} \to \tilde{B}$. Therefore
the inclusion $\tilde{P}_{(\tilde{B},\tilde{\mathcal{A}})} \subset 
\cotan \tilde{B}$ induces a homomorphism
\begin{equation*}
  \chi: \hom_{\Z}(\mathrm{C}_2(\tilde{B}); \Z^n) \to
  \hom_{\Z}(\mathrm{C}_2(\tilde{B});
  \hom_{\Z}(\mathrm{C}_1(\tilde{B});\R)),
\end{equation*}
\noindent
which maps the subgroup
$\hom_{\Z[\pi]}(\mathrm{C}_2(\tilde{B}); \Z^n)$  to $\pi$-equivariant
homomorphisms.

\begin{cor} \label{cor:s1}
The following diagram commutes
\begin{equation}
\label{eq:step}
\xymatrix{\hom_{\Z[\pi]}(\mathrm{C}_2(\tilde{B}); \Z^n)
  \ar@{^{(}->}[r]^-{\iota} \ar[d]^-{\chi} &
  \hom_{\Z}(\mathrm{C}_2(\tilde{B}); \Z^n) \ar[d]^-{\chi} \\ 
\hom_{\Z[\pi]}(\mathrm{C}_2(\tilde{B});
\hom_{\Z}(\mathrm{C}_1(\tilde{B});\R)) \ar@{^{(}->}[r]^-{\iota} & 
\hom_{\Z}(\mathrm{C}_2(\tilde{B}); \hom_{\Z}(\mathrm{C}_1(\tilde{B});\R))} 
\end{equation}
\end{cor}
\begin{proof}
The result follows from the commutativity of the diagram in equation
(\ref{eq:commdiag}). 
\end{proof}

\begin{rk}\label{rk:module_structure}
  In order for $\hom_{\Z[\pi]}(\mathrm{C}_2(\tilde{B});
  \hom_{\Z}(\mathrm{C}_1(\tilde{B});\R))$ to be a $\Z[\pi]$-module,
  fix the action of $\pi$ on $\mathrm{C}_2(\tilde{B})$ to be on the
  left, while the action on $\mathrm{C}_1(\tilde{B})$ to be on the
  right, so that $\hom_{\Z}(\mathrm{C}_1(\tilde{B});\R)$ is a left
  $\Z[\pi]$-module. While it is a technical point, it is important for the proof of Lemma
  \ref{lemma:fun}.
\end{rk}

Let $c$ be a cocycle representing the Chern class $[c]$ of
the almost Lagrangian fibration. The differential $D$ on the complex
$$\hom_{\Z[\pi]}(\mathrm{C}_*(\tilde{B}); 
\hom_{\Z}(\mathrm{C}_*(\tilde{B});\R))$$
\noindent 
is given by the graded sum of the differentials $\partial$ and
$\delta$ on the complexes $\mathrm{C}_*(\tilde{B})$ and 
$\hom_{\Z}(\mathrm{C}_*(\tilde{B});\R)$ respectively. Then   
\begin{equation*}
D \chi (c) =0,
\end{equation*}
\noindent
since $\chi(c)$ takes values in \emph{closed} 1-forms. The chain isomorphism (cf. \cite{dold})
\begin{equation*}
\Psi : \hom_{\Z}(\mathrm{C}_*(\tilde{B});
\hom_{\Z}(\mathrm{C}_*(\tilde{B});\R)) \to
\hom_{\Z}(\mathrm{C}_*(\tilde{B}) \otimes_{\Z}
\mathrm{C}_*(\tilde{B});\R)
\end{equation*}
\noindent
defined by
\begin{equation*}
\Psi (g) (x \otimes y) = (f(y))(x)
\end{equation*}
\noindent
for all $g \in \hom_{\Z}(\mathrm{C}_*(\tilde{B});
\hom_{\Z}(\mathrm{C}_*(\tilde{B});\R))$, $x \in \mathrm{C}_*(\tilde{B})$ and $y \in
\mathrm{C}_*(\tilde{B})$, preserves the $\Z[\pi]$-module structures
defined in Remark \ref{rk:module_structure} and, thus, defines a chain 
isomorphism
\begin{equation*}
  \Psi: \hom_{\Z[\pi]}(\mathrm{C}_*(\tilde{B});
\hom_{\Z}(\mathrm{C}_*(\tilde{B});\R)) \to
\hom_{\Z}(\mathrm{C}_*(\tilde{B}) \otimes_{\Z [\pi]}
\mathrm{C}_*(\tilde{B});\R).
\end{equation*}
\noindent
Therefore there is a commutative diagram
\begin{equation}
  \label{eq:25}
  \xymatrix{\hom_{\Z[\pi]}(\mathrm{C}_*(\tilde{B});
    \hom_{\Z}(\mathrm{C}_1(\tilde{B});\R)) \ar@{^{(}->}[r]^-{\iota}
    \ar[d]^-{\Psi} & 
    \hom_{\Z}(\mathrm{C}_*(\tilde{B});
    \hom_{\Z}(\mathrm{C}_1(\tilde{B});\R)) \ar[d]^-{\Psi} \\ 
    \hom_{\Z}(\mathrm{C}_1(\tilde{B}) \otimes_{\Z [\pi]}
    \mathrm{C}_*(\tilde{B});\R) \ar@{^{(}->}[r]^-{\iota} & 
    \hom_{\Z}(\mathrm{C}_1(\tilde{B}) \otimes_{\Z} \mathrm{C}_*(\tilde{B});\R).}
\end{equation}
\noindent
In light of the above commutative diagram, $\Psi \circ
\chi (c)$ can be seen as an equivariant element of
$\hom_{\Z}(\mathrm{C}_*(\tilde{B}) \otimes
\mathrm{C}_*(\tilde{B});\R)$. Consider the composition
\begin{equation*}
\Phi: \hom_{\Z}(\mathrm{C}_*(\tilde{B}) \otimes \mathrm{C}_*(\tilde{B});\R)
\to \hom_{\Z}(\mathrm{C}_*(\tilde{B} \times \tilde{B});\R) \to
\hom_{\Z}(\mathrm{C}_*(\tilde{B});\R),  
\end{equation*}
\noindent
where the first homomorphism is
induced by the Alexander-Whitney map 
$$A : C_*(\tilde{B} \times
\tilde{B}) \to C_*(\tilde{B}) \otimes C_*(\tilde{B})$$
\noindent
and the second by the diagonal map
$$ \tilde{\Delta}: \tilde{B} \to \tilde{B} \times \tilde{B}. $$
\noindent
It is standard that these maps can be made into equivariant simplicial
chain maps (cf. \cite{spanier}), \textit{i.e.} they can be chosen to
preserve the underlying $\Z[\pi]$-module structures. Therefore there is
a well-defined map
\begin{equation}
  \label{eq:29}
  \Phi \circ \Psi \circ \chi : \hom_{\Z[\pi]}(\mathrm{C}_*(\tilde{B}); \Z^n) \to
  \hom_{\Z}(\mathrm{C}_{*+1}(\tilde{B});\R),
\end{equation}
\noindent
which descends to cohomology. This is the candidate for $\mathcal{D}$. Note that the map $\Phi
\circ \Psi \circ \chi$ on standard cochains on $\tilde{B}$ is simply
given by the cup product by definition of the various homomorphisms
involved (cf. \cite{white}). 

\begin{defn}\label{defn:twisted_cup_product}
  The map of equation \eqref{eq:29} is called the \emph{twisted cup product} on $\am$.
\end{defn}

The following lemma proves that the image of the twisted cup product
lies in the $\pi$-\emph{invariant} forms on $\tilde{B}$,
\textit{i.e.} it defines a form on $B$. 

\begin{lemma} \label{lemma:fun}
Let $c$ be a cocycle representing the Chern class $[c]$ of the fixed
almost Lagrangian fibration $\tn \hookrightarrow M \to \am$. For all $\gamma \in \pi$,
\begin{equation*}
\gamma \cdot (\Phi \circ \Psi \circ \chi (c)) =
\Phi \circ \Psi \circ \chi (c) 
\end{equation*}
\end{lemma}
\begin{proof}
To simplify notation, set $\chi (c) = f$. Fix $\gamma \in \pi$
and let $z$ be a singular $3$-simplex taking values in $\tilde{B}$. Then 
\begin{equation}
\label{eq:18}
\begin{split}
(\Phi \circ \Psi)(f) (z) & = (\Psi (f))((z \circ \lambda^1) \otimes_{\Z \pi} (z \circ \mu^2)) \\
& = f(z \circ \mu^2)(z \circ \lambda^1),  
\end{split}
\end{equation}
\noindent
where the first equality follows from noticing that $\Phi$ is obtained
by composing the pullbacks of the diagonal map and the
Alexander-Whitney map on $\tilde{B}$. The map $\lambda^1: \Delta^1 \to \Delta^3$ is defined on
vertices of the $1$-simplex by $\lambda^1 e_i = e_i$ for $i=0,1$, while
$\mu^2 : \Delta^2 \to \Delta^3$ is defined on vertices of the
$2$-simplex by $\mu^2 e_j = e_{1+j}$ for $j=0,1,2$ (cf. \cite{spanier}).
On the other hand,
\begin{equation*}
\begin{split}
(\gamma \cdot (\Phi \circ \Psi)(f)) (z) & =
(\Phi \circ \Psi)(f) (\gamma \cdot z) \\  
& = (\Psi (f))(((\gamma \cdot z) \circ \lambda^1) \otimes_{\Z \pi}
((\gamma \cdot z) \circ \mu^2)). 
\end{split}
\end{equation*}
\noindent
By definition
\begin{equation}
\label{eq:simpl}
(\gamma \cdot z) \circ \lambda^1 (t)= \gamma ((z \circ \lambda^1)
(t))= ((z \circ \lambda^1) (t))\cdot \gamma^{-1} 
\end{equation}
\noindent
where the last equality follows from the fact that $C_1(\tilde{B})$ is
a \emph{right} $\Z [\pi]$-module (cf. Remark \ref{rk:module_structure}). Similarly, 
\begin{equation}
\label{eq:simpl2}
(\gamma \cdot z) \circ \mu^2 (t)= \gamma ((z \circ \mu^2) (t))= \gamma
\cdot ((z \circ \mu^2)(t)) .
\end{equation}
\noindent
Combining equations (\ref{eq:simpl}) and (\ref{eq:simpl2})
\begin{equation}
\label{eq:tada}
\begin{split}
(\Psi (f))(((\gamma \cdot z) \circ \lambda^1) \otimes_{\Z \pi}
((\gamma \cdot z) \circ \mu^2)) & = (\Psi (f))((z \circ
\lambda^1)\cdot \gamma^{-1} \otimes_{\Z \pi} (\gamma \cdot(z \circ
\mu^2))) \\ 
& = (\Psi (f))((z \circ \lambda^1) \otimes_{\Z \pi} (z \circ \mu^2)) \\
& = f(z \circ \mu^2)(z \circ \lambda^1),
\end{split}
\end{equation}
\noindent
where the penultimate equality follows by properties of the tensor
product. Comparing equations \eqref{eq:18} and (\ref{eq:tada}) yields
the required result.  
\end{proof}

Lemma \ref{lemma:fun} shows that
$\Phi \circ 
\Psi \circ \chi$ defines a homomorphism of cohomology groups
$$\coho^2(B;\Z^n_{\mathfrak{l}^{-T}}) \to \cohor. $$
\noindent
It remains to show that it coincides with the map
$\mathcal{D}$. Let $c$ be a cocycle representing the Chern class $[c]$
of the almost Lagrangian fibration $\tn \hookrightarrow M \to \am$; by
Remark \ref{rk:cc}, considering $c$ as a standard cocycle on
$\tilde{B}$ yields a cocycle representing $q^*[c]$, \textit{i.e.} the
Chern class of the almost Lagrangian bundle obtained by pulling back
along the universal covering $q$. The commutative diagrams of Corollary \ref{cor:s1} and equation 
\eqref{eq:25} imply that
\begin{equation}
\label{eq:monster}
\Psi \circ \Phi \circ \chi (c) = \Psi \circ \Phi \circ \chi (\iota(c)),
\end{equation}
\noindent
where
$$ \iota: \hom_{\Z[\pi]}(\mathrm{C}_2(\tilde{B}); \Z^n) \hookrightarrow
\hom_{\Z}(\mathrm{C}_2(\tilde{B}); \Z^n) $$
\noindent
denotes inclusion. The homomorphism 
$$ \Psi \circ \Phi :
\hom_{\Z}(\mathrm{C}_2(\tilde{B});\hom_{\Z}(\mathrm{C}_1(\tilde{B});\R))
\to \hom_{\Z}(\mathrm{C}_3(\tilde{B});\R) $$
\noindent
equals cup product by definition. Therefore, on the level of
cohomology groups, it equals the Dazord-Delzant homomorphism
$\tilde{\mathcal{D}}$ for the integral affine manifold $(\tilde{B},
\tilde{\mathcal{A}})$ by Theorem \ref{thm:main}. Since $\iota$ induces
the pullback map $q^*$ on cohomology, it follows that
\begin{equation}
\label{eq:equality}
\Phi \circ \Psi \circ \chi = \tilde{\mathcal{D}} \circ q^* :
\coho^2(B;\Z^n_{\mathfrak{l}^{-T}}) \to \coho^3(\tilde{B};\R).
\end{equation}

\begin{lemma} \label{lemma:comm}
The following diagram commutes
\begin{equation*}
\xymatrix{\coho^2(B;\Z^n_{\mathfrak{l}^{-T}}) \ar[r]^-{\mathcal{D}}
  \ar[d]^-{q^*} & \cohor \ar[d]^-{q^*} \\ 
\mathrm{H}^2(\tilde{B}; \zn) \ar[r]^-{\tilde{\mathcal{D}}} &
\mathrm{H}^3(\tilde{B};\R)}  
\end{equation*}
\end{lemma}
\begin{proof}
There is a commutative diagram of sheaves
\begin{equation}
\label{eq:commsheaves}
\xymatrix{\mathcal{P}_{\am} \ar[d]^-{q^*} \ar[r]^-{\jmath} & \mathcal{Z}^1
  (B) \ar[d]^-{q^*} \\ 
\tilde{\mathcal{P}}_{(\tilde{B},\tilde{\mathcal{A}})} \ar[r]^-{\jmath}
& \mathcal{Z}^1(\tilde{B})} 
\end{equation}
\noindent
where $\mathcal{P}_{\am}, \mathcal{\tilde{P}}_{(\tilde{B},\tilde{\mathcal{A}})}$
($\mathcal{Z}^1(B),\mathcal{Z}^1(\tilde{B})$) are the sheaves of
sections of the period lattice bundles $P_{\am} \to B$,
$\tilde{P}_{(\tilde{B},\tilde{\mathcal{A}})} \to 
\tilde{B}$ (sheaves of closed sections of the cotangent bundles
$\cotan B \to B$, $\cotan \tilde{B} \to \tilde{B}$) respectively. The
commutativity of the above diagram follows from equation
\eqref{eq:commdiag}. Equation 
(\ref{eq:commsheaves}) induces a 
commutative diagram of cohomology groups 
\begin{equation*}
\xymatrix{\coho^2(B;\mathcal{P}_{\am}) \ar[r]^-{\mathfrak{D}} \ar[d]^-{q^*} &
  \mathrm{H}^2(B;\mathcal{Z}^1(B)) \ar[d]^-{q^*} \\ 
\mathrm{H}^2(\tilde{B}; \tilde{\mathcal{P}}_{(\tilde{B},\tilde{\mathcal{A}})})
\ar[r]^-{\tilde{\mathfrak{D}}} &
\mathrm{H}^2(\tilde{B};\mathcal{Z}^1(\tilde{B})).} 
\end{equation*}
\noindent
There is a commutative ladder of short exact sequences of sheaves
\begin{equation}
  \label{eq:30}
  \xymatrix{0 \ar[r] & \R \ar[r] \ar[d]^-{q^*} &
    \mathcal{C}^{\infty}(B) \ar[r]^-{\de} \ar[d]^-{q^*} &
    \mathcal{Z}^1(B) \ar[r] \ar[d]^-{q^*} & 0 \\
    0 \ar[r] & \R \ar[r] & \mathcal{C}^{\infty}(\tilde{B})
    \ar[r]^-{\de} & \mathcal{Z}^1(\tilde{B}) \ar[r] &0;}
\end{equation}
\noindent
since both sheaves of smooth functions $\mathcal{C}^{\infty}(B)$ and
$\mathcal{C}^{\infty}(\tilde{B})$ are fine, it follows that there are
isomorphisms
\begin{equation*}
  \begin{split}
    \coho^2(B;\mathcal{Z}^1(B)) &\cong \coho^3(B;\R) \\
    \coho^2(\tilde{B};\mathcal{Z}^1(\tilde{B})) &\cong \coho^3(\tilde{B};\R)
  \end{split}
\end{equation*}
\noindent
which commute with the natural maps $q^*$ induced by the universal
covering $q: \tilde{B} \to B$. Applying these isomorphisms to the commutative diagram of
equation \eqref{eq:30}, the result follows.
\end{proof}

Lemma \ref{lemma:comm} states that $\tilde{\mathcal{D}} \circ q^* =
q^* \circ \mathcal{D}$. It follows from equation (\ref{eq:equality})
that 
\begin{equation}
\label{eq:almostthere}
\Phi \circ \Psi \circ \chi = \tilde{\mathcal{D}}
\circ q^* = q^* \circ \mathcal{D} 
\end{equation}
\noindent
Lemma \ref{lemma:fun} shows that the image of the map
$\Phi \circ \Psi \circ \chi$ lies precisely in
the subgroup of $\pi$-invariant closed $3$-forms on $\tilde{B}$. Such
forms are in $1-1$ correspondence with cohomology classes of $3$-forms
on $B$. Denote by $[\Phi \circ \Psi \circ
\chi(c)]_B$ the cohomology class of the $3$-form on $B$ obtained from
$\Phi \circ \Psi \circ \chi(c)$. Therefore the
following theorem, which contains the main result of the paper, holds. 

\begin{thm} \label{thm:equimain}
If $c$ denotes a cocycle representing $[c] \in \cohot$
\begin{equation*}
[\Phi \circ \Psi \circ \chi(c)]_B = \mathcal{D}[c]
\end{equation*}
\noindent
for all $[c] \in \cohot$. In other words, the Dazord-Delzant
homomorphism $\mathcal{D}$ is given by taking the twisted cup product.  
\end{thm}

\begin{rk}
  It is not necessary to use the universal cover $\tilde{B}$ in the
  above discussion. It is enough to consider an integral affine
  covering $(\hat{B},\hat{\mathcal{A}})$ of $\am$  
  corresponding to the normal subgroup $\ker \mathfrak{l}^{-T}$. By
  construction, $\hat{\mathcal{A}}$ has trivial linear holonomy and so the above
  considerations can be applied. More generally, Lemma \ref{lemma:comm}
  can also be applied to any integral affine regular covering $(\bar{B},\bar{\mathcal{A}})$ of $\am$ to show
  an analogous relation between the maps $\mathcal{D},
  \bar{\mathcal{D}}$. 
\end{rk}

The above discussion provides an explicit algorithm to compute
$\mathcal{D}$ for any integral affine manifold $\am$ with linear
holonomy $\mathfrak{l}$. Given an almost Lagrangian fibration $\tn
\hookrightarrow M \to \am$ with Chern class $[c]$, let $c$ be an equivariant cocycle representing
$[c]$. Fix a cell decomposition of $B$ (which is assumed here to be a
CW complex without loss of generality) and a $\mathfrak{l}^{-T}$-equivariant global frame
$\{\tilde{\theta}^1, \ldots, \tilde{\theta}^n\}$ for the embedding of
the period lattice bundle $\tilde{P}_{(\tilde{B},\tilde{\mathcal{A}})} \to \tilde{B}$ of the pullback
fibration $\tilde{M} \to \tilde{B}$. Note that the cell decomposition of
$B$ induces a $\pi_1(B)$-equivariant cell decomposition of
$\tilde{B}$; moreover, $\mathrm{C}_i(\tilde{B})$ is a free
$\Z[\pi_1(B)]$-module with basis given by choosing a representative of each
of the $i$-cell in $B$. To each $2$-cell $e^2_r$ of
$B$, $c$ associates a vector $\mathbf{v}_r=(v^1_r, \ldots,
v^n_r)$. Using the map $\chi$, identify the vector $\mathbf{v}_r$ with
$$ v^1_r \tilde{\theta}^1 + \ldots + v^n_r \tilde{\theta}^n $$
\noindent
and for each $2$-cell $e^2_r$ in $B$ choose a Kronecker dual
$\epsilon^2_r \in \hom(\mathrm{C}_2(B);\Z)$. For each $r$, the
pullback of $\epsilon^2_r
: \mathrm{C}_2(B) \to \Z$ to $\tilde{B}$ defines an element
$\tilde{\epsilon}^2_r \in 
\hom_{\Z}(\mathrm{C}_2(\tilde{B});\Z)$. The sum  
$$\sum\limits_{r,l} v^l_r \tilde{\epsilon}^2_r \cup \tilde{\theta}^l $$
\noindent
is then the cocycle $\Phi \circ \Psi \circ \chi
(c)$, whose cohomology class (as a $3$-form with real values on $B$)
is $\mathcal{D}[c]$. 

\section{Examples} \label{sec:applications}
In this section the subgroup $R$ of realisable Chern classes is
computed for some integral affine manifolds so as to illustrate the
algorithm to compute the Dazord-Delzant homomorphism $\mathcal{D}$
of Section \ref{sec:general}. The following theorem,
quoted here without proof, is used in each example to prove that the
manifolds under consideration are indeed integral affine (cf. \cite{goldman}).

\begin{thm}\label{thm:inheritance}
  Let $\am$ be an integral affine manifold and $\Gamma$ a group acting
  on $\am$ by integral affine diffeomorphisms such that the quotient
  $B /\Gamma$ is a manifold. Then $B/\Gamma$ inherits an integral
  affine structure $\mathcal{A}_{\Gamma}$ from $\am$.
\end{thm}

\subsection*{$\tth$ with standard integral affine structure}
This example has also been considered in \cite{sepe_ex}, where
different methods are used. Consider the
integral affine manifold $\rt/\zt$, where the action of $\Z^3$ on
$\R^3$ is by translations along the standard cocompact lattice. These
are integral affine diffeomorphisms of $\R^3$ with respect to its standard
integral affine structure and thus the quotient is an integral affine
manifold, which is denoted by $\R^3/\Z^3$ for notational ease. Its
linear monodromy $\mathfrak{l}$ is trivial. Let $(x^1, x^2,x^3)$ be
the standard (integral affine) 
coordinates on $\rt$ inducing integral affine coordinates on $\tth$,
which are also denoted by $(x^1, x^2, x^3)$ by abuse of
notation. Let $F \hookrightarrow M \to \R^3/\Z^3$ be an
almost Lagrangian fibration with trivial monodromy (the
fibre is denoted by $F$ to avoid confusion with the base). The period
lattice bundle $P_{\R^3/\Z^3} \to \R^3/\Z^3$ associated to this fibration is trivial
and the embedding of equation \eqref{eq:24} determines a global
framing of $\cotan \R^3/\Z^3 \cong \R^3/\Z^3 \times \R^3$. Fix the framing to be
$\{ \de x^1, \de x^2, \de x^3\}$. Let $\{ [\beta_1], [\beta_2], [\beta_3]\}$
denote the standard ordered basis of $\mathrm{H}^2(\R^3/\Z^3; \Z)$. Under the
homomorphism
\begin{equation}\label{eq:26}
  \begin{split}
    \mathrm{H}^2(\R^3/\Z^3; \Z) & \to \mathrm{H}^2(\R^3/\Z^3;\Z) \otimes_{\Z} \R
    \cong \mathrm{H}^2(\R^3/\Z^3;\R) \\
    [A] & \mapsto [A] \otimes_{\Z} 1
  \end{split}  
\end{equation}
\noindent
the above basis is mapped to the standard ordered basis $\{[\de x^1
\wedge \de x^2], [\de x^2\wedge \de x^3], [\de x^3 \wedge \de x^1]\}$
of $\mathrm{H}^2(\R^3/\Z^3;\R)$. In order to follow the conventions set in
\cite{sepe_ex}, fix the ordered frame of the period lattice bundle to be
$\{ \theta^1=\de x^3, \theta^2=\de x^1, \theta^3=\de x^2\}$. Let 
$$[c] = (c_{kl}) \in \mathrm{H}^2(\R^3/\Z^3;\zt) \cong
\mathrm{H}^2(\R^3/\Z^3; \Z) \otimes \Z\langle \de x^3, \de x^1, \de x^2
\rangle$$
\noindent
be the Chern class of the fibration $M \to \R^3/\Z^3$, so that
$$ [c] = \sum\limits_{l,r=1}^3c_{lr}\, [\beta_l] \otimes \theta^r.$$
\noindent
The result of Theorem \ref{thm:main} shows that
$$ \mathcal{D}[c] = \sum\limits_{l,r=1}^3c_{lr}\, [\beta^{\R}_l \cup \theta^r]$$
\noindent 
where $\beta^{\R}_l$ is a cocycle representing the image of the
cohomology class $[\beta_l]$ under the map of equation
\eqref{eq:26}. Therefore,
$$ \mathcal{D}[c] = (c_{11} + c_{22} + c_{33}) \, [\de x^1 \wedge \de
x^2 \wedge \de x^3] \in \mathrm{H}^3(\R^3/\Z^3;\R)$$
\noindent 
where the isomorphism between singular cohomology with real
coefficients and de Rham cohomology has been used tacitly. Since $[\de
x^1 \wedge \de x^2 \wedge \de x^3]$ is a generator of
$\mathrm{H}^3(\R^3/\Z^3;\R)$,
$$\mathcal{D}[c] = 0 \quad \Leftrightarrow \quad c_{11} + c_{22} + c_{33} =
0,$$
\noindent
as shown in \cite{sepe_ex}.

\subsection*{The $3$-dimensional Heisenberg manifold}
The $3$-dimensional Heisenberg manifold is an interesting example of
an integral affine manifold as its fundamental group is not abelian,
but nilpotent. The relation between affine manifolds with nilpotent
fundamental groups and affine geometry has already been investigated
in \cite{hirsch}, although much less is known in the context of
solvable fundamental groups.\\

Let 
$$ \mathrm{N}_3(\R) = \Bigg\{
\begin{pmatrix}
  1 & x^1 & x^3 \\
  0 & 1 & x^2 \\
  0 & 0 & 1
\end{pmatrix}\,:\, x^1,x^2,x^3 \in \R \Bigg\}$$
\noindent
denote the set of unipotent upper-triangular $3 \times 3$ matrices with real
coefficients and let
$$ \Gamma = \mathrm{N}_3(\Z) = \Bigg\{
\begin{pmatrix}
  1 & a^1 & a^3 \\
  0 & 1 & a^2 \\
  0 & 0 & 1
\end{pmatrix}\,:\, a^1,a^2,a^3 \in \Z \Bigg\}$$
\noindent
denote the lattice of unipotent upper-triangular $3 \times 3$ matrices with
integral coefficients. The $3$-dimensional Heisenberg manifold is
given by
$$ H = \Gamma \backslash \mathrm{N}_3(\R) $$
\noindent
where $\Gamma$ acts on $\mathrm{N}_3(\R)$ by left
multiplication. $\mathrm{N}_3(\R) \cong \rt$ is an integral
affine manifold; let $x^1,x^2,x^3 : \mathrm{N}_3(\R) \to \R$ be globally
defined integral affine coordinates on $\mathrm{N}_3(\R)$. The lattice
$\Gamma \subset \mathrm{N}_3(\R)$ is generated by the matrices
\begin{equation}
  \label{eq:12}
  \alpha =
  \begin{pmatrix}
    1 & 1 & 0 \\
    0 & 1 & 0 \\
    0 & 0 & 1
  \end{pmatrix}, \quad \beta =
  \begin{pmatrix}
    1 & 0 & 0 \\
    0 & 1 & 1 \\
    0 & 0 & 1  
  \end{pmatrix}, \quad \gamma =
  \begin{pmatrix}
    1 & 0 & 1 \\
    0 & 1 & 0 \\
    0 & 0 & 1
  \end{pmatrix}.
\end{equation}
\noindent
Denote a point in $\mathrm{N}_3(\R)$ by its integral affine
coordinates $(x^1, x^2, z^3)$ and let $L : \Gamma \to
\mathrm{Aut}(\mathrm{N}_3(\R))$ denote the representation of $\Gamma$
induced by left multiplication. Note that
\begin{equation}
  \label{eq:13}
  \begin{split}
    \alpha \cdot
    \begin{pmatrix}
      x^1 \\
      x^2 \\
      x^3
    \end{pmatrix}
    & =
    \begin{pmatrix}
      1 & 0 & 0 \\
      0 & 1 & 1 \\
      0 & 0 & 1
    \end{pmatrix}
    \begin{pmatrix}
      x^1 \\
      x^2 \\
      x^3
    \end{pmatrix}
    +
    \begin{pmatrix}
      1 \\
      0 \\
      0
    \end{pmatrix} \\
    \beta \cdot \begin{pmatrix}
      x^1 \\
      x^2 \\
      x^3
    \end{pmatrix}
    & =
    \begin{pmatrix}
      1 & 0 & 0 \\
      0 & 1 & 0 \\
      0 & 0 & 1
    \end{pmatrix}
    \begin{pmatrix}
      x^1 \\
      x^2 \\
      x^3
    \end{pmatrix}
    +
    \begin{pmatrix}
      0 \\
      1 \\
      0
    \end{pmatrix} \\
    \gamma \cdot \begin{pmatrix}
      x^1 \\
      x^2 \\
      x^3
    \end{pmatrix}
    & =
    \begin{pmatrix}
      1 & 0 & 0 \\
      0 & 1 & 0 \\
      0 & 0 & 1
    \end{pmatrix}
    \begin{pmatrix}
      x^1 \\
      x^2 \\
      x^3
    \end{pmatrix}
    +
    \begin{pmatrix}
      0 \\
      0 \\
      1
    \end{pmatrix}
  \end{split}
\end{equation}
\noindent
Equation \eqref{eq:13} shows that the image $L(\Gamma)$ lies in the
group $\mathrm{Aff}_{\Z}(\mathrm{N}_3(\R))$ of integral affine
isomorphisms of $\mathrm{N}_3(\R)$. In particular, Theorem
\ref{thm:inheritance} implies that Heisenberg manifold $H$ inherits an integral affine structure from its
universal cover $\mathrm{N}_3(\R)$. Denote this integral affine
manifold also by $H$ and, by abuse of notation, let $(x^1,x^2,x^3)$ be
its integral affine coordinates. The linear monodromy $\mathfrak{l} :
\Gamma = \pi_1(H) \to \glthr$ is defined on the generators of equation \eqref{eq:12} by
\begin{equation}
  \label{eq:14}
  \mathfrak{l}(\alpha) =
  \begin{pmatrix}
    1 & 0 & 0 \\
    0 & 1 & 1 \\
    0 & 0 & 1
  \end{pmatrix}, \quad \mathfrak{l}(\beta) = I = \mathfrak{l}(\gamma)
\end{equation}
\noindent
Let $M \to H$ be an almost Lagrangian fibration with monodromy $\mathfrak{l}^{-T}$, and let $q: \mathrm{N}_3(\R) \to H$ be the
universal covering map. The pullback fibration $q^* M \to
\mathrm{N}_3(\R)$ has trivial monodromy, since $\mathrm{N}_3(\R)$ is
simply connected. In particular, the period lattice bundle $\tilde{P}_{\mathrm{N}_3(\R)}
\to \mathrm{N}_3(\R)$ into $\cotan
\mathrm{N}_3(\R)$ determines a choice of global frame of $\cotan
\mathrm{N}_3(\R)$ by closed forms. The homomorphism of equation
\eqref{eq:24} gives the ordered frame $\{\tilde{\theta}^1=\de x^3,
\tilde{\theta}^2=\de x^1, \tilde{\theta}^3 =\de
x^2\}$ of $\cotan \mathrm{N}_3(\R)$. This frame is \emph{equivariant}
with respect to the standard symplectic lift of the action of $\pi_1(H) = \Gamma$ on
$\mathrm{N}_3(\R)$ by deck transformations to the cotangent bundle
$\cotan \mathrm{N}_3(\R)$.  \\

The fundamental group $\Gamma$ of $H$ admits the following presentation in terms
of generators $\alpha, \beta, \gamma$ of equation \eqref{eq:12}
$$ \Gamma = \langle \alpha, \beta, \gamma \,:\, \alpha \beta = \gamma
\beta \alpha,\, \alpha \gamma = \gamma \alpha, \, \beta \gamma =
\gamma \beta \rangle$$
\noindent
$H$ admits a CW-decomposition with one $0$-cell $e^0$, three $1$-cells
$e_1^1,e^1_2, e^1_3$, three $2$-cells $e_1^2,e^2_2, e^2_3$ and one
$3$-cell $e^3$. This decomposition induces a $\Z[\Gamma]$-equivariant
CW decomposition on $\tilde{H} = \mathrm{N}_3(\R)$ as in the
previous example. Fix a basis for the graded $\Z[\Gamma]$-module
$\mathrm{C}_*(\tilde{H}) = \bigcup_{l\geq 1} \mathrm{C}_l(\tilde{H})$,
denoted by $e^0, e^1_1,e^1_2, e^1_3,e^2_1,e^2_2, e^2_3,e^3$ 
by abuse of notation. The boundary maps in the
$\Z[\Gamma]$-equivariant chain complex are
\begin{equation}
  \label{eq:15}
    \begin{split}
    \partial_1(e_1^1) &= (\alpha - 1)e^0 \\
    \partial_1(e_2^1) &= (\beta - 1)e^0 \\
    \partial_1(e_3^1) &= (\gamma - 1)e^0 \\
    \partial_2(e_1^2) &= (1- \gamma \beta)e^1_1 + (\alpha-
    \gamma)e^1_2 - e^1_3\\
    \partial_2(e_2^2) &= (1- \gamma)e^1_2 - (1- \beta)e^1_3 \\
    \partial_2(e_3^2) &= (1- \gamma)e^1_1 + (\alpha-1)e^1_3 \\
    \partial_3(e^3) &= (\gamma-1)e^2_1 + (\alpha-\gamma)e^2_2 + (1-\gamma\beta)e^2_3,
  \end{split}
\end{equation}
\noindent
and can be obtained from the above presentation of the group by
noticing that the augmented chain complex $\mathrm{C}_* (\tilde{H})
\to \Z$ is a free $\Z[\Gamma]$-resolution of the trivial module
$\Z$. The twisted cohomology group $\mathrm{H}^2(H; \zt_{\mathfrak{l}^{-T}})$ is
calculated as follows. Let $\phi = (\phi_1,\phi_2,\phi_3):
\mathrm{C}_2(\tilde{B}) \to \zt$ be 
a $\Z[\Gamma]$-equivariant map. Then
\begin{equation*}
  \begin{split}
    (\delta_3 \phi)(e^3) & = \phi (\partial_3 e^3) \\ 
    & = (\mathfrak{l}^{-T}(\gamma)-I)\phi(e^2_1) +(\mathfrak{l}^{-T}(\alpha) -
    \mathfrak{l}^{-T}(\gamma))\phi(e^2_2)+(I -
    \mathfrak{l}^{-T}(\gamma) \mathfrak{l}^{-T}(\beta))\phi(e^2_3) \\
    & = \begin{pmatrix}
      0 & 0 & -1 \\
      0 & 0 & 0 \\
      0 & 0 & 0
    \end{pmatrix}
    \phi(e^2_2)
  \end{split}
\end{equation*}
\noindent
Thus $\phi$ defines a closed $2$-cocycle if and only if
\begin{equation*}
  \phi_3(e^2_2) = 0,
\end{equation*}
\noindent
and 
\begin{equation}
  \label{eq:20}
  \ker \delta \cong (\Z \oplus \Z \oplus \Z) \oplus (\Z \oplus
  \Z \oplus 0) \oplus (\Z \oplus \Z \oplus \Z), 
\end{equation}
\noindent
where each summand in brackets corresponds to the contribution of a
$2$-cell. Suppose now that $\phi = \delta \psi \in
\hom_{\Z[\Gamma]}(\mathrm{C}_2(\tilde{H}); \zt)$, then 
\begin{equation*}
  \begin{split}
    (\delta \psi)(e^2_1) & =(I -\mathfrak{l}^{-T}(\gamma)\mathfrak{l}^{-T}(\beta)) \phi(e^1_1) +
    (\mathfrak{l}^{-T}(\alpha) - \mathfrak{l}^{-T}(\beta))\phi(e^1_2) - \phi(e^1_3) \\
    (\delta \psi)(e^2_2) & =(I -\mathfrak{l}^{-T}(\gamma)) \phi(e^1_2) -
    (I - \mathfrak{l}^{-T}(\beta))\phi(e^1_2)\\
    (\delta \psi)(e^2_3) & =(I -\mathfrak{l}^{-T}(\gamma)) \phi(e^1_1) +
    (\mathfrak{l}^{-T}(\alpha)-I)\phi(e^1_3) \\
  \end{split}
\end{equation*}
\noindent
which implies that
\begin{equation}
  \label{eq:21}
  \mathrm{im}\, \delta  \cong (\Z \oplus \Z \oplus \Z) \oplus (0
  \oplus 0 \oplus 0) \oplus (0 \oplus 0\oplus 0).
\end{equation}
\noindent
Equations \eqref{eq:20} and \eqref{eq:21} imply that
\begin{equation}
  \label{eq:22}
  \mathrm{H}^2(H; \zt_{\mathfrak{l}^{-T}}) = \ker \delta / \mathrm{im}\, \delta
  \cong (0 \oplus 0 \oplus 0) \oplus (\Z 
  \oplus \Z \oplus 0) \oplus (\Z \oplus \Z \oplus \Z).
\end{equation}
\noindent
Thus it is possible to calculate the image of the map $\mathcal{D}$
using Theorem \ref{thm:equimain}. Let $\epsilon^2_1, \epsilon^2_2,
\epsilon^2_3 \in \hom (C_2(H); \Z)$ be Kronecker duals to
$e^2_1, e^2_2, e^2_3$. As explained at the end of Section
\ref{sec:general}, these induce elements $\tilde{\epsilon}^2_l \in
\hom_{\Z}(C_2(\tilde{H}); \Z)$ for $l = 1,2,3$. \\

Let $[c] \in \mathrm{H}^2(H; \zt_{\mathfrak{l}^{-T}})$ be represented by a cocycle
\begin{equation*}
  c = \sum\limits_{l,r=1}^3 c_{lr} \tilde{\epsilon}^2_l \otimes \tilde{\theta}^r.
\end{equation*}
Thus a $3$-cocycle on $\tilde{H}$ representing $\mathcal{D}[c]$ is given
by
\begin{equation}\label{eq:28}
  \sum\limits_{l,r=2}^3 c_{lr} \tilde{\epsilon}^2_l \cup \tilde{\theta}^r.
\end{equation}
\noindent
The calculations to determine $\mathrm{H}^2(H;\zt_{\mathfrak{l}^{-T}})$ show that
it is possible to assume that $c_{1r},c_{23} = 0$ for
$r=1,2,3$. Explicit representatives of $\tilde{\epsilon}^2_2,
\tilde{\epsilon}^2_3$ on a fundamental domain for $H$ in $\tilde{H}$
in terms of differential forms are given by $\de x^2 \wedge \de x^3,
\de x^3 \wedge \de x^1$ respectively. Therefore the cocycle of
equation \eqref{eq:28} can be represented in terms of differential
forms by 
\begin{equation*}
  (c_{22} + c_{33})\, \de x^1 \wedge \de x^2
  \wedge \de x^3.
\end{equation*}
\noindent
The form $\de x^1 \wedge \de x^2 \wedge \de x^3$ is under the action
of $\Gamma$ on $\tilde{H}$ and so it descends to $H$; its cohomology
class $[\de x^1 \wedge \de x^2 \wedge \de x^3]_{H}$ generates
$\mathrm{H}^3(H;\R)$ and 
therefore 
\begin{equation*}
  \mathcal{D}[c] = (c_{22} + c_{33})\, [\de x^1 \wedge \de x^2
  \wedge \de x^3]_{H}.
\end{equation*}

\subsection*{Mapping torus of $-\text{Id.}: \ttw \to \ttw$}
Let $B$ denote the mapping torus of the involution $-\text{Id.}: \ttw
\to \ttw$. This manifold admits an integral affine structure, since it
can be realised as follows. Consider the $\Z/2$-action on $\rt/\zt$ defined by 
\begin{equation}
  \label{eq:8}
  \zeta \cdot (x^1, x^2, x^3) = (x^1 + 1/2, -x^2, -x^3)  
\end{equation}
\noindent
where $(x^1, x^2, x^3)$ are integral affine coordinates on $\tth$ as
in the previous example, and $\zeta \in \Z/2$ is a generator. The
action is free and by integral affine 
transformations and thus the quotient $B = (\R^3/\Z^3)/(\Z/2)$ is an integral
affine manifold. Let $\am$ denote this integral affine manifold and let
$\mathfrak{l}$ be its linear holonomy. \\

Let $F \hookrightarrow M \to \am$ be an almost Lagrangian fibration with monodromy $\mathfrak{l}^{-T}$ and
Chern class $[c] \in \mathrm{H}^2(B;\zt_{\mathfrak{l}^{-T}})$. Denote by $ q:\R^3/\Z^3
\to B$ the integral affine double covering map given by the
$\Z/2$-action of equation \eqref{eq:8}. The pullback almost 
Lagrangian fibration $q^* M \to \R^3/\Z^3$ has trivial monodromy by construction. The
period lattice bundle $\hat{P}_{\R^3/\Z^3} \to  
\R^3/\Z^3$ associated to the pullback fibration $q^*M \to \R^3/\Z^3$ is trivial
and its embedding into $\cotan \R^3/\Z^3$ induces a choice of global
ordered frame $\{\tilde{\theta}^1 = \de x^3, \tilde{\theta}^2 = \de
x^1, \tilde{\theta}^3= \de x^2\}$. \\

The fundamental group of $B$ can be presented using generators and
relations as
$$ \pi = \pi_1(B) =\langle \alpha,\beta, \gamma:\, \beta\gamma = \gamma\beta, \,
\alpha = \beta\alpha\beta , \, \alpha=\gamma\alpha\gamma
\rangle. $$
\noindent
The subgroup
generated by $\alpha^2, \beta, \gamma$ is isomorphic to $\zt$ and
corresponds to the fundamental group of the double cover $ \R^3/\Z^3 \to
B$. The linear monodromy of the integral affine manifold $B$ is given
by the map $\mathfrak{l} : \pi_1(B) \to \glthr$ defined on the generators by
\begin{equation}
  \label{eq:2}
  \mathfrak{l}(\alpha) = \text{diag}(-1,1,-1), \quad \mathfrak{l}(\beta) = I, \quad
  \mathfrak{l} (\gamma) = I,
\end{equation}
\noindent
where the coordinate system used is $(x^3,x^1,x^2)$, in line with the
above choice of ordered frame of $\cotan \R^3/\Z^3$. The manifold $B$ admits a CW decomposition with one
$0$-cell $e^0$, three $1$-cells $e^1_1,e^1_2,e^1_3$, three $2$-cells
$e^2_1,e^2_2,e^2_3$ and one $3$-cell $e^3$. This decomposition induces
a $\pi$-equivariant cell decomposition on the universal cover
$\tilde{B} = \rt$
$$ \rt = \bigcup_{g \in \pi_1(B)} (e^0_g \cup e^1_{1,g} \cup e^1_{2,g} \cup
e^1_{3,g} \cup e^2_{1,g} \cup e^2_{2,g} \cup e^2_{3,g} \cup e^3_g)$$
\noindent
Fix the following $\Z[\pi]$-basis for the module
$\mathrm{C}_i(\tilde{B})$
\begin{equation}
  \label{eq:3}
  \begin{cases}
  e^0_{g_0} & \text{if } i=0 \\
  e^1_{1,g_0},e^1_{2,g_0},e^1_{3,g_0} & \text{if } i=1 \\
  e^2_{1,g_0},e^2_{2,g_0},e^2_{3,g_0} & \text{if } i=2 \\
  e^3_{g_0} & \text{if } i=3
  \end{cases}
\end{equation}
\noindent
where $g_0\in \pi$ is some fixed element. For notational ease,
suppress the dependence on 
$g_0$. The boundary maps
$$ \partial_i : \mathrm{C}_i(\tilde{B}) \to \mathrm{C}_{i-1}(\tilde{B})$$
\noindent
of the $\Z[\pi]$-equivariant chain complex of $B$ are given on the
basis of equation \eqref{eq:3} by
\begin{equation}
  \label{eq:4}
  \begin{split}
    \partial_1(e_1^1) &= (\alpha \beta \gamma - 1)e^0 \\
    \partial_1(e_2^1) &= (\beta - 1)e^0 \\
    \partial_1(e_3^1) &= (\gamma - 1)e^0 \\
    \partial_2(e_1^2) &= (1- \beta)e^1_1 - (1+\alpha \gamma)e^1_2 \\
    \partial_2(e_2^2) &= (1- \gamma)e^1_2 - (1- \beta)e^1_3 \\
    \partial_2(e_3^2) &= (1- \gamma)e^1_1 - (1+\alpha\beta)e^1_3 \\
    \partial_3(e^3) &= (1- \gamma)e^2_1 - (1-\alpha)e^2_2 + (1-\beta)e^2_3
  \end{split}
\end{equation}
\noindent
using standard methods in algebraic topology. The corresponding
$\Z[\pi]$-equivariant cochain complex with values in $\zt$ arises when
applying the equivariant $\hom_{\Z[\pi]}(\,.\,;\zt)$ functor to the chain complex
defined above. In particular, for the calculation of
$\mathrm{H}^2(B;\zt_{\mathfrak{l}^{-T}})$, the 
relevant maps are
$$ \xymatrix@1{\hom_{\Z[\pi]}(\mathrm{C}_1(\tilde{B});\zt)
  \ar[r]^-{\delta_2} & \hom_{\Z[\pi]}(\mathrm{C}_2(\tilde{B});\zt)
  \ar[r]^-{\delta_3} & \hom_{\Z[\pi]}(\mathrm{C}_3(\tilde{B});\zt)} $$
\noindent
The twisted cohomology group $\mathrm{H}^2(B;\zt)$ is given by
$$ \mathrm{H}^2(B; \zt_{\mathfrak{l}^{-T}}) = \ker \delta_3/\mathrm{im}\, \delta_2. $$
\noindent
Let $\phi : \mathrm{C}_2(\tilde{B}) \to \zt$ be a twisted cocycle,
then
\begin{equation*}
  \begin{split}
    0 & = (\delta_3 \phi)(e^3) = \phi (\partial_3 e^3) \\ 
    & = (I -
    \mathfrak{l}^{-T}(\gamma))\phi(e^2_1) -(I -
    \mathfrak{l}^{-T}(\alpha))\phi(e^2_2)+(I -
    \mathfrak{l}^{-T}(\beta))\phi(e^2_3) \\
    & = \mathrm{diag}(2,0,2) \phi(e^2_2)
  \end{split}
\end{equation*}
\noindent
where the last equality follows from the definition of $\mathfrak{l}^{-T}$ in
equation \eqref{eq:2}. Thus $\phi$ is a twisted cocycle if and only if
\begin{equation}
  \label{eq:6}
  \phi_1(e^2_2) = 0 = \phi_3(e^2_2)
\end{equation}
\noindent
where $\phi = (\phi_1,\phi_2,\phi_3)$. Suppose now that $\phi$ is a
twisted coboundary, \textit{i.e.} that it lies in the image of
$\delta_2$, so that $\phi = \delta_2 \psi$. Using the boundary maps of
equation \eqref{eq:4}, it is possible to obtain the following
conditions on $\phi$
\begin{equation}
  \label{eq:7}
  \phi_2(e^2_1) = 2k_1, \quad \phi_2(e^2_3) = 2k_2 \\
\end{equation}
\noindent
for $k_1,k_2 \in \Z$. Using equations \eqref{eq:6} and \eqref{eq:7},
it follows that 
$$\mathrm{H}^2(B;\zt_{\mathfrak{l}^{-T}}) \cong (\Z \oplus \Z/2 \oplus \Z) \oplus
(0 \oplus \Z \oplus 0) \oplus (\Z \oplus \Z/2 \oplus \Z). $$

Choose Kronecker duals $\epsilon^2_l \in \hom(\mathrm{C}_2(B);\Z)$ for
each $2$-cells $e^2_l$ in $B$. These induce elements $\tilde{\epsilon}^2_l \in
\hom_{\Z}(\mathrm{C}_2(\R^3/\Z^3);\Z)$. Let $c$ be an equivariant cocycle
representing $[c] \in 
\mathrm{H}^2(B; \zt_{\mathfrak{l}^{-T}})$ on $\R^3/\Z^3$. It takes the form
\begin{equation*}
  c = \sum\limits_{l,r=1}^3 c_{lr} \, \tilde{\epsilon}^2_l \cup \tilde{\theta}^r
\end{equation*}
\noindent
Explicit representatives of $\tilde{\epsilon}^2_1,
\tilde{\epsilon}^2_2, \tilde{\epsilon}^2_3$ on a fundamental domain for $B$ in $\R^3/\Z^3$
in terms of differential forms are given by $\de x^1 \wedge \de x^2,
\de x^2 \wedge \de x^3, \de x^3 \wedge \de x^1$
respectively. Therefore the above cocycle can be represented in terms
of differential forms on $\R^3/\Z^3$ as
\begin{equation*}
    (c_{11}+ c_{22} + c_{33})\, \de x^1 \wedge \de x^2 \wedge \de x^3.
\end{equation*}
Since $ \de x^1 \wedge \de x^2 \wedge \de x^3$ is invariant under the
action of $\pi$ by deck transformations on $\R^3/\Z^3$, it descends to a
closed $3$-form on $B$ whose cohomology class $[ \de x^1 \wedge \de x^2
  \wedge \de x^3]_B$ generates $\mathrm{H}^3(B;\R)$. Thus
$$ \mathcal{D}[c] = (c_{11}+c_{22}+c_{33})\, [\de x^1 \wedge \de x^2 \wedge
  \de x^3]_B. $$


\section{Conclusion}
The above description of the map $\mathcal{D}$ further emphasises the
relation between (almost) Lagrangian fibrations and integral affine
manifolds. While being useful to carry out the classification of
Lagrangian fibrations over higher dimensional manifolds, the map
$\mathcal{D}$ should also reveal information about integral affine
manifolds. For instance, a natural question to ask is

\begin{qn}
  Is it possible to choose the forms $\tilde{\theta}^1, \ldots,
  \tilde{\theta}^n$ so that the image of the map $\mathcal{D}$ lies in
  a cohomology group $\mathrm{H}^3(B;S)$, where $S$ is some finite
  ring extension of the integers?
\end{qn}

The examples of Section \ref{sec:applications} show that this might
indeed be the case. In order to answer this and many other related
questions, it is necessary to investigate what it means for a manifold to be integrally
affine, which is central to the
study of the topology and symplectic geometry of completely integrable
Hamiltonian systems.


\bibliographystyle{abbrv}
\bibliography{mybib}
\end{document}